\documentclass[11pt]{amsart}
\usepackage{amssymb,amsmath,epsfig,mathrsfs, enumerate, xparse, mathtools}
\usepackage[pagewise]{lineno}
\usepackage[nodisplayskipstretch]{setspace}
\usepackage{graphicx}\usepackage[normalem]{ulem}
\usepackage{fancyhdr}
\usepackage[margin=2.5cm]{geometry}
\usepackage{hyperref}
\hypersetup{
 colorlinks   = true,
 urlcolor     = blue,
 linkcolor    = blue,
 citecolor   = red ,
 bookmarksopen=true
}

\theoremstyle{plain}
\newtheorem{thrm}{Theorem}[section]
\newtheorem{lemma}[thrm]{Lemma}

\allowdisplaybreaks
\begin{document}
\newcommand{\sn}{\mathbb{S}^{n-1}}
\newcommand{\SL}{\mathcal L^{1,p}( D)}
\newcommand{\Lp}{L^p( Dega)}
\newcommand{\py}{  \partial_{x_{n+1}}^a}
\newcommand{\La}{\mathscr{L}_a}
\newcommand{\CO}{C^\infty_0( \Omega)}
\newcommand{\Rn}{\mathbb R^n}
\newcommand{\Rm}{\mathbb R^m}
\newcommand{\R}{\mathbb R}
\newcommand{\Om}{\Omega}
\newcommand{\Hn}{\mathbb H^n}
\newcommand{\aB}{\alpha B}
\newcommand{\eps}{\ve}
\newcommand{\BVX}{BV_X(\Omega)}
\newcommand{\p}{\partial}
\newcommand{\IO}{\int_\Omega}
\newcommand{\bG}{\boldsymbol{G}}
\newcommand{\bg}{\mathfrak g}
\newcommand{\bz}{\mathfrak z}
\newcommand{\bv}{\mathfrak v}
\newcommand{\Bux}{\mbox{Box}}
\newcommand{\e}{\ve}
\newcommand{\X}{\mathcal X}
\newcommand{\Y}{\mathcal Y}
\newcommand{\W}{\mathcal W}
\newcommand{\la}{\lambda}
\newcommand{\vf}{\varphi}
\newcommand{\rhh}{|\nabla_H \rho|}
\newcommand{\Ba}{\mathcal{B}_\beta}
\newcommand{\Za}{Z_\beta}
\newcommand{\ra}{\rho_\beta}
\newcommand{\n}{\nabla}
\newcommand{\vt}{\vartheta}
\newcommand{\its}{\int_{\{y=0\}}}

\numberwithin{equation}{section}

\newcommand{\RN} {\mathbb{R}^N}
\newcommand{\Sob}{S^{1,p}(\Omega)}
\newcommand{\Dxk}{\frac{\partial}{\partial x_k}}
\newcommand{\Co}{C^\infty_0(\Omega)}
\newcommand{\Je}{J_\ve}
\newcommand{\beq}{\begin{equation}}
\newcommand{\bea}[1]{\begin{array}{#1} }
\newcommand{\eeq}{ \end{equation}}
\newcommand{\ea}{ \end{array}}
\newcommand{\eh}{\ve h}
\newcommand{\Dxi}{\frac{\partial}{\partial x_{i}}}
\newcommand{\Dyi}{\frac{\partial}{\partial y_{i}}}
\newcommand{\Dt}{\frac{\partial}{\partial t}}
\newcommand{\aBa}{(\alpha+1)B}
\newcommand{\GF}{\psi^{1+\frac{1}{2\alpha}}}
\newcommand{\GS}{\psi^{\frac12}}
\newcommand{\HFF}{\frac{\psi}{\rho}}
\newcommand{\HSS}{\frac{\psi}{\rho}}
\newcommand{\HFS}{\rho\psi^{\frac12-\frac{1}{2\alpha}}}
\newcommand{\HSF}{\frac{\psi^{\frac32+\frac{1}{2\alpha}}}{\rho}}
\newcommand{\AF}{\rho}
\newcommand{\AR}{\rho{\psi}^{\frac{1}{2}+\frac{1}{2\alpha}}}
\newcommand{\PF}{\alpha\frac{\psi}{|x|}}
\newcommand{\PS}{\alpha\frac{\psi}{\rho}}
\newcommand{\ds}{\displaystyle}
\newcommand{\Zt}{{\mathcal Z}^{t}}
\newcommand{\XPSI}{2\alpha\psi \begin{pmatrix} \frac{x}{\left< x \right>^2}\\ 0 \end{pmatrix} - 2\alpha\frac{{\psi}^2}{\rho^2}\begin{pmatrix} x \\ (\alpha +1)|x|^{-\alpha}y \end{pmatrix}}
\newcommand{\Z}{ \begin{pmatrix} x \\ (\alpha + 1)|x|^{-\alpha}y \end{pmatrix} }
\newcommand{\ZZ}{ \begin{pmatrix} xx^{t} & (\alpha + 1)|x|^{-\alpha}x y^{t}\\
     (\alpha + 1)|x|^{-\alpha}x^{t} y &   (\alpha + 1)^2  |x|^{-2\alpha}yy^{t}\end{pmatrix}}
\newcommand{\norm}[1]{\lVert#1 \rVert}
\newcommand{\ve}{\varepsilon}
\newcommand{\D}{\operatorname{div}}
\newcommand{\G}{\mathscr{G}}
\newcommand{\sa}{\langle}
\newcommand{\da}{\rangle}

\title[An observation on eigenfunctions, etc.]{An observation on eigenfunctions of the Laplacian}

\author{Agnid Banerjee}
\address{Tata Institute of Fundamental Research\\
Centre For Applicable Mathematics \\ Bangalore-560065, India}\email[Agnid Banerjee]{agnidban@gmail.com}

\author{Nicola Garofalo}
\address{University of Padova, Italy}\email[Nicola Garofalo]{nicola.garofalo@unipd.it}

\thanks{A. Banerjee  is supported in part by SERB Matrix grant MTR/2018/000267 and by Department of Atomic Energy,  Government of India, under
project no.  12-R \& D-TFR-5.01-0520. N. Garofalo is supported in part by a Progetto SID (Investimento Strategico di Dipartimento): ``Aspects of nonlocal operators via fine properties of heat kernels", University of Padova, 2022. He has also been partially supported by a Visiting Professorship at the Arizona State University}

%
%
%

\maketitle

\begin{abstract}
In his seminal 1943 paper F. Rellich proved that, in the complement of a cavity $\Om = \{x\in \Rn\mid |x|>R_0\}$, there exist no nontrivial solution $f$ of the Helmholtz equation $\Delta f = - \la f$, when $\la>0$, such that $\int_{\Om} |f|^2 dx < \infty$. In this note we generalise this result by showing that if $\int_{\Om} |f|^p dx < \infty$ for some $0<p\le \frac{2n}{n-1}$, then $f\equiv 0$. This result is sharp since for any $p> \frac{2n}{n-1}$, eigenfunctions do exist in $\Om$. 
\end{abstract}

\section{Introduction}

At the basis of Von Neumann's axiomatic formulation of quantum mechanics is the spectral theorem. Because of this, the problem of ruling out eigenvalues embedded in the continuous spectrum of the time-independent Schr\"odinger operator has occupied a central position in mathematical physics. In his famous paper \cite{Rel}, F. Rellich proved that, if for some $R_0>0$, $f\not\equiv 0$ solves the Helmholtz equation 
\begin{equation}\label{helmo}
\Delta f = - \la f,
\end{equation}
in $\Om = \{x\in \Rn\mid |x|>R_0\}$, then $f\not\in L^2(\Om)$. It is well-known that this result plays an instrumental role in ruling out the existence of positive eigenvalues for $H = - \Delta + V(x)$, when for instance $V$ is compactly supported in $\Rn$. Furthermore, in the same paper Rellich also proved that if for $|x|\to \infty$ one has
\begin{equation}\label{repoint}
f(x) =  o(|x|^{-\frac{n-1}2}),
\end{equation}
then $f\equiv 0$ in $\Om$. In the pointwise category, the assumption \eqref{repoint} is sharp since for any $\la >0$ the function 
\begin{equation}\label{f}
f(x) = \int_{\mathbb S^{n-1}} e^{i \sqrt \la \sa x,\omega\da} d\sigma(\omega) = c(n,\la) |x|^{- \frac{n-2}2} J_{\frac{n-2}{2}}(\sqrt \la |x|),
\end{equation}
is a spherically symmetric solution of \eqref{helmo} in the whole $\Rn$. By the classical asymptotic behaviour of the Bessel functions, one has for $|x|\to \infty$,
\[
J_{\frac{n-2}{2}}(\sqrt \la |x|) = O(|x|^{-1/2}).
\]
It follows from \eqref{f} that 
\begin{equation}\label{O}
f(x) = O(|x|^{-\frac{n-1}2}),
\end{equation}
thus violating \eqref{repoint}. On the other hand, the asymptotic behaviour \eqref{O} implies that $f\in L^p(\Om)$ if and only if $p>\frac{2n}{n-1}$ (note that, since such exponent is $>2$, it is clear that $f\not\in L^2(\Om)$, in accordance with Rellich's cited result).

This leads to the following natural question. Do nontrivial solutions of the Helmholtz equation \eqref{helmo} exist such that for some $0< p \le \frac{2n}{n-1}$, and such that for, say, $\Om^\star = \{x\in \Rn\mid |x|>R_0+4\}$, one has
\begin{equation}\label{p}
\int_{\Om^\star} |f|^p dx <\infty?
\end{equation}

The aim of this note is to prove that the answer is no.

\begin{thrm}\label{T:helmop}
Let $f$ be a solution to \eqref{helmo}, for some $\la>0$. If for some $0< p \le \frac{2n}{n-1}$, the function $f$ satisfies \eqref{p}, then it must be $f\equiv 0$ in $\Om$.
\end{thrm}

The proof of Theorem \ref{T:helmop} is given in Section \ref{S:proof}. In closing, we mention that a generalisation of Rellich's cited result to partial differential equations with constant coefficients was obtained by Littman \cite{Lit}. Important extensions to Schr\"odinger operators were found by Kato \cite{Ka}, Agmon \cite{A} and Simon \cite{RS}. We also signal the paper \cite{GS} which, using exclusively Rellich identities, generalised the results of Kato, Agmon and Simon to the degenerate Baouendi-Grushin operator 
\[
\Delta_x + |x|^{2\alpha} \Delta_y,\ \ \ \ \ \ x\in \Rm, y\in \R^k,\ \alpha>0.
\]

\medskip

\noindent \textbf{Acknowledgment:} We thank Matania Ben-Artzi, Ermanno Lanconelli, Eugenia Malinnikova and Yehuda Pinchover for kindly providing helpful feedback on the question raised in this note.


\section{Proof of Theorem \ref{T:helmop}}\label{S:proof}

We begin by recalling, without proof, the following classical $L^\infty$ estimate for harmonic functions. We mention that its proof is a trivial consequence of the mean-value theorem when $1\le p<\infty$, but it is more delicate when $0<p<1$.  

\begin{lemma}\label{L:har}
Let $\Delta u = 0$ in a ball $B(x,4r)\subset \Rn$ (alternatively, one could take a nonnegative subharmonic function). For every $0<p<\infty$, there exists a universal constant $C(n,p)>0$ such that
\[
\underset{\overline B(x,r)}{\max}\ |u| \le C(n,p) \left(\frac{1}{|B(x,2r)|} \int_{B(x,2r)} |u|^p dy\right)^{\frac 1p}.
\]
\end{lemma}

In what follows, if $\Om = \{x\in \Rn\mid |x|>R_0\}$, we denote $\Om^\star = \{x\in \Rn\mid |x|>R_0+4\}$. We will need the following auxiliary result. 

\begin{lemma}\label{L:infty}
Let $f$ be a solution to \eqref{helmo} in $\Om$, and suppose that $f$ satisfy \eqref{p} for some $p>0$. Then $f\in L^\infty(\Om^\star)$, and there exists $C(n,\la,p)>0$ such that
\begin{equation}\label{infty}
||f||_{L^\infty(\Om^\star)} \le C(n,\la,p) \left(\int_{\Om} |f|^p dy\right)^{\frac 1p}.
\end{equation}
\end{lemma}

\begin{proof}
In $\Om\times \R$, consider the function $u(x,y) = e^{\sqrt \la y} f(x)$, where $x\in \Om$, $y\in \R$. We have
\[
\Delta_{(x,y)} u = e^{\sqrt \la y} \Delta_x f + \la e^{\sqrt \la y} f(x)= 0.
\]
By Lemma \ref{L:har}, for any $x\in  \Om^\star$ we have
\begin{align*}
& \underset{\overline B((x,0),1)}{\max}\ |u| \le C(n,p) \left(\frac{1}{|B((x,0),2)|} \int_{B((x,0),2)} |u|^p dy dy_{n+1}\right)^{\frac 1p} 
\\
& = \overline C(n,p) \left(\int_{B((x,0),2)} |u|^p dy dy_{n+1}\right)^{\frac 1p}.
\end{align*}
Since the ball $B((x,0),2)\in\R^{n+1}$ is clearly contained in the cylinder $B(x,2) \times [-2,2]$, this estimate easily implies for some $C(n,p,\la)>0$,
\[
|f(x)| \le \underset{\overline B((x,0),1)}{\max}\ |u| \le C(n,p,\la) \left(\int_{B(x,2)} |f|^p dy\right)^{\frac 1p}\le C(n,p,\la) \left(\int_{\Om} |f|^p dy\right)^{\frac 1p}.
\]
By the arbitrariness of $x\in \Om^\star$ we conclude that 
\eqref{infty} does hold.

\end{proof}

We can now give the
\begin{proof}[Proof of Theorem \ref{T:helmop}]
In \cite{Rel} Rellich proved that if $f$ solves \eqref{helmo} in $\Om$, then there exist $R_1>R_0$ sufficiently large, and a constant $M = M(n,\la)>0$, such that for every $R>R_1$ one has
\begin{equation}\label{rellich}
\int_{R<|x|<2R} |f|^2 dx \ge M R.
\end{equation} 
Suppose now that $2\le p \le \frac{2n}{n-1}$. Applying H\"older inequality to \eqref{rellich}, we obtain
\begin{equation}\label{rellich2}
M R \le C(n,p) \left(\int_{R<|x|<2R} |f|^p dx\right)^{\frac 2p} R^{n(1-\frac 2p)}.
\end{equation}
When $p = \frac{2n}{n-1}$, we have $n(1-\frac 2p)=1$, and \eqref{rellich2} gives
\[
0<M  \le C(n,p) \left(\int_{R<|x|<2R} |f|^p dx\right)^{\frac 2p}.
\]
However, this inequality is contradictory with the assumption \eqref{p} for such $p$, since by letting $R\to\infty$, the right-hand side converges to $0$ by Lebesgue dominated convergence. We are left with analysing the case $0<p<2$. In such range, under the hypothesis $f\not\equiv 0$ in $\Om$, Rellich's inequality \eqref{rellich} trivially implies for $R>R_0 + 10$, 
\begin{equation}\label{rellich3}
M R \le ||f||^{2-p}_{L^\infty(\Om^\star)} \int_{R<|x|<2R} |f|^p dx.
\end{equation}
By the real-analyticity (or the unique continuation property) of the Helmholtz equation, we must have $f\not\equiv 0$ in $\Om^\star$, and therefore $||f||_{L^\infty(\Om^\star)}>0$. On the other hand, if $f$ satisfies \eqref{p}, then by Lemma \ref{L:infty} we also have  $||f||_{L^\infty(\Om^\star)} < \infty$. 
Letting $R\to \infty$ in \eqref{rellich3} we thus reach a contradiction again.

\end{proof}

\end{document}